\documentclass[12pt]{article}
\usepackage{amsmath}
\usepackage{amssymb}
\usepackage{amsthm}
\usepackage{graphicx} 
\usepackage[legalpaper, margin=1in]{geometry}

\newtheorem{theorem}{Theorem}[section]
\newtheorem{corollary}[theorem]{Corollary}
\newtheorem{lemma}[theorem]{Lemma}
\newtheorem{proposition}[theorem]{Proposition}

\theoremstyle{definition}
\newtheorem{definition}[theorem]{Definition}
\newtheorem{question}[theorem]{Question}
\newtheorem{remark}[theorem]{Remark}

\theoremstyle{remark}

\newcommand{\R}{\mathbb{R}}
\newcommand{\N}{\mathbb{N}}
\newcommand{\Q}{\mathbb{Q}}
\newcommand{\Z}{\mathbb{Z}}

\newcommand{\Fix}{\mathrm{Fix}}
\newcommand{\Acc}{\mathrm{Acc}}
\newcommand{\T}{\mathsf{T}}

\title{
Classification complexity of chaotic systems
} 
\author{
Benjamin Vejnar\footnote{https://orcid.org/0000-0002-2833-5385} \footnote{
Supported by the grant GACR 24-10705S
}\\
Department of Mathematical Analysis\\ 
Faculty of Mathematics and Physics, Charles University\\
Prague, Czechia\\
E-mail: vejnar@karlin.mff.cuni.cz}

\begin{document}
\maketitle

\renewcommand{\thefootnote}{}



\renewcommand{\thefootnote}{\arabic{footnote}}
\setcounter{footnote}{0}

\begin{abstract}
In this paper, we deal with the classification complexity of continuous (Devaney) chaotic systems in dimensions $0,1$ and $\infty$ using the framework of invariant descriptive set theory. We identify the complexity in dimensions $0$ and $\infty$, while in dimension $1$ we get some partial results.

More precisely, we prove the topological conjugacy relation of invertible chaotic systems on the Hilbert cube (resp. on all compact metric spaces) has the same complexity as (i.e. is Borel bireducible with) the universal orbit relation induced by a Polish group.
As a consequence, this answers a recent question asked by L. Ding. We also prove that the topological conjugacy relation of invertible chaotic systems on the Cantor space
has the same complexity as the universal relation induced by the group $S_\infty$. This answers a recent question by M. Foreman.
Some non-trivial bounds on the classification complexity of chaotic systems on the interval and on the circle are also obtained. Namely, the lower bound is the Vitali equivalence relation, and the upper bound is the equality of countable sets of reals. This especially implies that the relation is Borel. However, the exact complexity remains unknown.

\end{abstract}

\section{Introduction} 

Classification of dynamical systems is one of the central problems in the field.
Perhaps the first nontrivial results of this kind go back to Poincaré, who classified transitive rotations of the circle by using the rotation number as a complete invariant \cite[Theorem 7.1.9]{BrinStuck}. Ornstein classified Bernoulli shifts using entropy as a complete invariant \cite{Ornstein}. However, it is rarely the case that a real number may serve as a \emph{natural} complete invariant for more involved systems.
Invariant descriptive set theory is a natural tool that allows us to formalize the concept of classification using more complicated objects as natural invariants \cite{Gao}.

Recently, the study of classification complexity (up to measurable or continuous conjugacy) on various classes of dynamical systems, from the perspective of invariant descriptive set theory, becomes of great interest (see, e.g. \cite{ForemanRudolphWeiss, ForemanDynamical, Buzzietc, Kunde, deka2025bowensproblem32conjugacy}). 
We focus on the classification of dynamical systems up to conjugacy by a homeomorphism. 
Clemens proved that the conjugacy of subshifts is as difficult as the universal Borel equivalence relation with countable classes \cite{Clemens}.
Hjorth identified the complexity of interval invertible systems and observed that the two-dimensional case is more complex \cite{Hjorth}.
Bruin and the author found that the conjugacy of interval systems is as difficult as classifying countable graphs up to isomorphism, whereas the conjugacy of systems based on the Hilbert cube is much more complex \cite{BruinVejnar}.


The aim of this paper is to focus on chaotic systems.
A chaotic system (in the sense of Devaney) is a classical dynamical system for which orbits of periodic points form a dense subset of the hyperspace of compact sets. 
Until now, no attention was paid to the classification complexity of these systems. 
Using some notation from Section 2, the main results of our investigation dealing with chaotic systems are formulated as follows:
\begin{itemize}
\item[{[A]}]   
The classification complexity of chaotic invertible systems on the Cantor space is the same as the universal orbit relation induced by the permutation group $S_\infty$ (Theorem \ref{Cantor}).
\item[{[B]}]   
The classification complexity of chaotic systems on the interval or on the circle is bounded from below by the Vitali equivalence relation and bounded from above by the universal $F_{\sigma\delta}$ relation induced by the group $S_\infty$
(Theorems \ref{conjugacyinterval} and \ref{transitivecircle} and Corollary \ref{cor:chaoticinterval}).
\item[{[C]}]  
The classification complexity of chaotic invertible systems on the Hilbert cube is the same as the universal orbit relation induced by a Polish group (Theorem \ref{Hilbertcube}).
\end{itemize}

Let us note that
[A] answers one half of a question by Matthew Foreman \cite[Open problem 15]{ForemanDynamical}), whereas
[C] yields Corollary \ref{allcompacta} which answers a question of Longyun Ding posed during the Nankai Logic seminar.
Since the result [B] does not give a concrete complexity level, it would be nice to know the answer to the following question.

\begin{question}
What is the exact complexity of the conjugacy relation of chaotic systems on the interval, resp. on the circle? 
\end{question}

This question could be possibly answered by one of the complexity levels mentioned in Theorem [B] or some natural equivalence relations in between them, e.g. the countable power of the Vitali relation, the universal Borel relation with countable classes, or its countable power (see \cite[p. 351]{Gao}). But there are also others.


Let us now switch to minimal systems.
Recently, Kaya proved that the equality of countable sets $=^+$ relation is Borel reducible to the conjugacy of minimal Cantor systems \cite{Kaya}.
In a fresh paper \cite{LiPeng}, Li and Peng proved that the conjugacy equivalence relation of minimal systems (on compact metric spaces) cannot be classified by countable structures, i.e. it is not Borel reducible to $E_{S_\infty}$, which is the universal orbit relation induced by an action of the group $S_\infty$.
A group of authors announced that the conjugacy relation of Cantor minimal systems is not Borel \cite{Kwietniak}.
All of these are only partial results that tend to find the exact complexity of conjugacy of minimal homeomorphisms (either on the Cantor space or on all compact metrizable spaces). However, the exact classification complexity of minimal systems on compact metric spaces as well as on the Cantor set is still unknown.

By a conjecture of Sabok, the conjugacy relation of minimal invertible systems is Borel bireducible with the universal orbit equivalence relation $E_{G_\infty}$ \cite{LiPeng}. 
When trying to prove the conjecture of Sabok, we need to be able to construct various minimal systems. Minimal systems often occur on homogeneous continua (or compacta). However, such spaces are rare and their classification up to homeomorphism is unknown \cite[Problem 1]{Ciesla}. In fact, only recently it has been proven that homogeneous continua form an analytic set \cite{Krupski}.
The very first construction of a minimal system supported by a non-homogeneous compact space comes from Floyd \cite{Floyd}. 
His idea was recently enhanced by lifting any minimal Cantor space homeomorphism to a minimal homeomorphism of a cantoroid (i.e. a compact space whose degenerate components are non-isolated points and altogether form a dense set) whose nondegenerate components are topological copies of a prescribed space obtained from an injective iterated function system \cite{DeeleyPutnamStrung}. However, iterated function systems have finite Hausdorff dimension, and thus also finite topological dimension and additionally the complexity of the homeomorphism equivalence relation of such spaces is generally supposed to be strictly below $E_{G_\infty}$ (see \cite{DudakVejnarFinitedim, GaoChangFinitedim}). 
Since we moreover lack enough injective iterated function systems (see \cite{Dumitru} for countably many examples), we decided to move to other type of cantoroids as base spaces.

Let assign to each continuum $K$ a cantoroid $X_K$ in which nondegenerate components are homeomorphic to $K$ and form a dense null sequence. In Theorem \ref{tamecantoroids} it is shown that the assignment $K\mapsto X_K$ is a Borel reduction and thus the homeomorphism relation of such cantoroids is Borel bireducible with $E_{G_\infty}$.
Moreover, it is known that every cantoroid admits a minimal map \cite{BalibreaDownarowiczHricSnohaSpitalsky}, so we could hope that it will be possible to choose a minimal homeomorphism $h_K$ on every cantoroid $X_K$ to get a Borel reduction $K\mapsto (K, h_K)$. However, by Theorem
\ref{tamecantoroidsminimal}, the conjugacy relation on tame cantoroids (excepting the Cantor space) is Borel. 
This contrasts with the fact that $E_{G_\infty}$ (even $E_{S_\infty}$) is a non-Borel relation. Consequently, tame cantoroids are not suitable base spaces to prove the conjecture by Sabok.

As being true in the chaotic case, one could hope that the complexity of conjugacy relation of minimal systems is witnessed by using one concrete space. Since the Hilbert cube has the fixed point property, it does not admit minimal maps. A natural guess is the infinite-dimensional torus $\mathbb S^{\N}$.

\begin{question}
What is the complexity of conjugacy relation of minimal invertible systems based on an infinite-dimensional torus?
\end{question}


When dealing with cantoroids, we prove the following complete result:

\begin{itemize}
    \item[{[D]}] The classification complexity of chaotic invertible systems on all cantoroids is that of $E_{G_\infty}$ (Theorem \ref{conjugacycantoroidschaotic}). 
\end{itemize}




\section{Preliminaries}

We denote by $\N, \Z, \Q, \R$ the sets $\{1,2,\dots\}$, the integers, rational numbers, and real numbers, respectively.

\subsection{Topology and descriptive set theory}
All spaces considered in this paper are topological metrizable and separable. A Polish space is a completely metrizable space.
A Borel set in a Polish space is any element of the sigma algebra generated by open sets. An analytic set is a continuous image of a Polish space (resp. an image of a Borel set with respect to a Borel measurable map).

A Cantor space is a compact non-empty zero-dimensional space which lacks isolated points. It is well known that every Cantor space is homeomorphic to the Cantor ternary set.
The Hilbert cube is the countable product $\prod_{n=1}^\infty[0,1]$ and is denoted by $Q$.
An absolute retract is a space homeomorphic to a retract of the Hilbert cube. A dendrite is an absolute retract with topological dimension one, or equivalently, a locally connected continuum that does not contain a topological copy of the circle as a subspace.
We denote by $\mathcal K(X)$ the space of all nonempty compact subsets of $X$ equipped with the Vietoris topology. See the books \cite{Kechris, Nadler} for more information.

\subsection{Dynamics}
A dynamical system is a pair $(X,T)$ where $X$ is a compact space and $T:X\to X$ is a continuous map. It is called invertible if $T$ is a homeomorphism. Two dynamical systems $(X,T)$ and $(Y, S)$ are called (topologically) conjugate if there is a homeomorphism $h:X\to Y$ such that $h\circ T=S\circ h$.
An orbit of a point $x\in X$ is the set $\{T^n(x):n\in\N\cup \{0\}\}$.
A dynamical system $(X, T)$ is called (topologically) transitive if for every pair $U, V$ of nonempty open sets there is $n\geq 0$ such that $T^n(U)\cap V\neq\emptyset$.
The system $(X, T)$ is called minimal if every point has a dense orbit.

Chaotic systems belong to the central notions of this paper. A dynamical system $(X, T)$ is called (Devaney) chaotic if for every finite cover $\mathcal U$ of $X$ by non-empty open sets, there is a periodic point $x\in X$ whose orbit meets every element of the cover $\mathcal U$. This is equivalent to saying that periodic orbits form a dense subset of $\mathcal K(X)$. It is also equivalent to the usual definition of (Devaney) chaotic map in the sense that it is a transitive map whose periodic points are dense.

\subsection{Invariant descriptive set theory}
\begin{definition}\label{DefinitionOfReduction}
Let $E$ and $F$ be equivalence relations on sets $X$ and $Y$, respectively.
A map $f \colon X \to Y$ is called a reduction from $E$ to $F$ if for every two points $x , x' \in X$ we have $x E x'$ iff $f(x) F f(x')$.
If $X$, $Y$ are Polish spaces and there is a Borel measurable reduction from $E$ to $F$, then we say that $E$ is Borel reducible to $F$ and write $E {\leq}_B F$. We say that $E$ is Borel bireducible with $F$ if $E {\leq}_B F$ and $F {\leq}_B E$.
\end{definition}

An equivalence relation that is Borel reducible to the equality of real numbers is called smooth.
The most simple non-smooth relation is the Vitali equivalence relation \[E_v=\{(x,y)\in\R^2:x-y\in\Q\}.\]
which is known to be Borel bireducible with the relation of eventual equality:
\[E_0=\{(x,y)\in\{0,1\}^{\N}:\exists n\forall m>n: x_m=y_m\}.\]
Given a class $\mathcal{C}$ of equivalence relations on Polish spaces and an element $E \in \mathcal{C}$, we say that $E$ is universal for $\mathcal{C}$ if $F {\leq}_B E$ for every $F \in \mathcal{C}$.
By the orbit equivalence relation induced by a group action $\varphi \colon G \times X \to X$ we understand the equivalence relation $E_G^X$ on $X$ given by
\[ x E_G^X x' \iff \exists \, g \in G : \varphi (g,x) = x' . \]

It is known that for every Polish group $G$ there exists an equivalence relation $E_G$ that is universal for the class of all orbit equivalence relations induced by Borel actions of $G$ on Polish spaces \cite[Theorem 5.1.8]{Gao}. Some of the most studied equivalence relations in the field of invariant descriptive set theory include $E_{S_{\infty}}$ and $E_{G_{\infty}}$, where $S_{\infty}$ stands for the symmetric group on $\N$ and $G_{\infty}$ stands for the universal Polish group. An equivalence relation on a Polish space is classifiable by countable structures if and only if it is Borel reducible to $E_{S_{\infty}}$  \cite[Theorem 2.39]{Hjorth}. The equivalence relation $E_{G_{\infty}}$ is universal for the class of all orbit equivalence relations induced by Borel actions of Polish groups on Polish spaces \cite[Theorem 5.1.9]{Gao}.
For a brief introduction to Borel reductions, see \cite{Foreman}.

Let us conclude that the most important complexity levels which we are using form a strictly increasing chain, where the two relations on the left-hand side are Borel sets, whereas the two relations on the right-hand side are complete analytic sets:
\[E_0 <_B E_{=^+}<_B E_{S_\infty} <_B E_{G_\infty}.\]

\section{Chaotic invertible systems}

The following proposition by Burgess is a special kind of a selection theorem and it will serve as a useful tool to complete the Borel coding argument in Theorem \ref{Hilbertcube}. Let us note that similar usage of the Burgess selection theorem is realized also in \cite[Proposition 4.9]{BruinVejnar} with the aim to prove that the conjugacy relation of all Hilbert cube homeomorphisms is Borel bireducible with $E_{G_\infty}$.

\begin{proposition}[{\cite{Burgess}, \cite[Corollary 5.4.12]{Gao}}]\label{Burgesstheorem}
Let $G$ be a Polish group acting in a Borel way on a  Polish space. If the induced orbit equivalence relation $E_G^X$ is smooth, then it has a Borel selector.
\end{proposition}

\begin{theorem}\label{Hilbertcube}
The conjugacy equivalence relation of invertible chaotic systems on the Hilbert cube is Borel bireducible with $E_{G_\infty}$.
\end{theorem}

\begin{proof}
Denote by $E$ the conjugacy relation of invertible chaotic systems on the Hilbert cube.
By \cite{KrupskiVejnar}, the homeomorphism relation $H$ of absolute retracts is Borel bireducible with $E_{G_\infty}$. We may suppose that the relation $H$ is defined only on nondegenerate spaces.
Moreover $E\leq_B E_{G_\infty}$ by \cite[Theorem 7.1]{LiPeng}.
Thus it is enough to prove that $H\leq_B E$.
Let $f(K)=(K^{\mathbb Z}, \sigma)$ be the shift of a countable infinite power of $K$, where $K$ is a non-degenerate absolute retract.
It is known that a countable infinite power of a non-degenerate absolute retract is homeomorphic to the Hilbert cube \cite[Corollary 8.1.2]{vanMill} and that the shift map is transitive and has a dense set of periodic points.
Thus, $(K^{\mathbb Z}, \sigma)$ is a chaotic dynamical system on a Hilbert cube. 

We argue that $f$ is a reduction of $H$ to $E$. Thus, fix two absolute retracts $K, L\subseteq Q$.
We want to prove that $K$ is homeomorphic to $L$ if and only if $(K^{\mathbb Z}, \sigma)$ and $(L^{\mathbb Z}, \sigma)$ are conjugate. The direct implication is triviality. On the other hand, suppose that $(K^{\mathbb Z}, \sigma_K)$ and $(L^{\mathbb Z}, \sigma_L)$ are conjugate by a homeomorphism $\varphi: K^{\mathbb Z}\to L^{\mathbb Z}$, i.e. $\varphi\circ\sigma_K = \sigma_L\circ\varphi$. It follows that $\varphi(Fix(\sigma_K))=Fix(\sigma_L)$. Moreover $Fix(\sigma_K)$ consists of constant sequences in $K^{\mathbb Z}$ and thus is homeomorphic to $K$. Similarly $Fix(\sigma_L)$ is homeomorphic to $L$.
Consequently $K$ is homeomorphic to $L$. Together, $K$ is homeomorphic to $L$ if and only if $f(K)$ is isomorphic to $f(L)$.

The proof is not yet completely done, since we got a reduction to chaotic homeomorphisms on different Hilbert cubes. However, this can be fixed as follows.
Consider the Polish group $G:=\mathcal H(Q)$ of all Hilbert cube homeomorphisms acting on the Polish space $X:=\mathcal I(Q, Q^{\Z})$ of all injective maps by composition. Under this action, two maps from $\mathcal I(Q, Q^{\Z})$ are orbit equivalent iff they have the same images. Moreover, the equivalence relation $E_G^X$ is closed in $X\times X$ which makes the relation $E_G^X$ smooth \cite[Proposition 5.4.7]{Gao}.
Thus, by applying Proposition \ref{Burgesstheorem}, we get a Borel map assigning to every copy $K^{\Z}\subseteq Q^{\Z}$ of the Hilbert cube a homeomorphism $i_K: Q\to K^{\Z}$.
Defining $g(K):= i_K^{-1}\circ f(K)\circ i_K$ we get that $g$ is the desired Borel reduction of $H$ to $E$.
Consequently, $E\sim_B E_{G_\infty}$.
\end{proof}

\begin{corollary}\label{allcompacta}
The conjugacy relation of transitive systems on compact spaces is Borel bireducible with $E_{G_\infty}$.
\end{corollary}

\begin{proof}
Let $E$ be the equivalence relation from the statement.
By Theorem \ref{Hilbertcube} we get $E_{G_\infty}\leq_B E$ and by \cite[Theorem 7.1]{LiPeng} we get that $E\leq_B E_{G_\infty}$. Thus $E$ is Borel bireducible with $E_{G_\infty}$.
\end{proof}

\begin{theorem}\label{Cantor}
The conjugacy relation of chaotic systems on the Cantor space is Borel bireducible with $E_{S_\infty}$.
\end{theorem}

\begin{proof}
The idea of the proof is very similar to the proof of Theorem \ref{Hilbertcube}. Let us denote by $E$ the conjugacy relation of chaotic Cantor space homeomorphisms.
First, by \cite{CamerloGao} it is known that conjugacy of Cantor maps is Borel bireducible with $E_{S_\infty}$.
Hence $E\leq_B E_{S_\infty}$.
It is also known that homeomorphism equivalence relation $H$ of closed subsets of the Cantor space is Borel bireducible with $E_{S_\infty}$ \cite{CamerloGao}.
We may easily suppose that $H$ is defined only for spaces containing at least two points.
Hence it is enough to show that $H\leq_B E$.

Let $f(K)=(K^{\Z}, \sigma_K)$ be the shift dynamical system.
Clearly $K^\Z$ is a zero-dimensional compact space without isolated points and thus it is homeomorphic to the Cantor space. The shift map is clearly transitive and has a dense set of periodic points.
As in the proof of Theorem \ref{Hilbertcube} we can argue that $f$ is a reduction and our coding argument could be transfered to a Borel reduction of $H$ to the conjugacy relation $E$ of chaotic homeomorphisms on a fixed Cantor space.
Consequently, $E\sim_B E_{S_\infty}$.
\end{proof}

\section{Minimal and chaotic invertible systems on cantoroids}

\begin{definition}
A cantoroid is a nonempty compact space without isolated points such that the union of all degenerate components is dense. A cantoroid is called tame if all nondegenerate components 
form a null sequence, i.e. the diameters tend to zero.
\end{definition}

\begin{theorem}\label{tamecantoroids}
The homeomorphism relation of tame cantoroids is Borel bireducible with $E_{G_\infty}$.
\end{theorem}

\begin{proof}
The homeomorphism relation of all compacta resp. all continua is Borel bireducible with $E_{G_\infty}$ \cite{Zielinski, ChangGao}.
Thus it is enough to find a Borel reduction of the homeomorphism relation on continua to the homeomorphism relation on tame cantoroids.
Let $K$ be a nondegenerate continuum.
By an inverse limit we construct a tame cantoroid whose all nondegenerate components are homeomorphic to $K$.
Let $X_1$ be a compactification of $\N$ with remainder $K$. 
Suppose that we have already constructed $X_1,\dots, X_n$, $n\geq 1$. Let $X_{n+1}$ be a space obtained from $X_n$ in such a way, that all isolated points of $X_n$ are replaced by a small clopen copy of $X_1$. Note that there are countably many isolated points in $X_n$, so all of these are replaced by a null sequence of copies of $X_1$. A map $f_n$ naturally  shrinks all the inserted copies of $X_1$ in $X_{n+1}$ to the corresponding isolated points in $X_n$.
Let $X_K$ be the inverse limit of $(X_n, f_n)_{n\in\N}$. Note that there are essentially two kinds of points $(x_n)_{n\in\N}$ in $X_K$: those, for which $x_n$ is contained eventually in some copy of $K$ and those, for which $x_n$ is always an isolated point in $X_n$. The first kind of points forms a countable set of copies of $K$ and the second kind of points forms a dense set of degenerate components.

Let us verify that the map $K\mapsto X_K$ is a reduction of the homeomorphism equivalence relation into itself.
Suppose first that $K$ is homeomorphic to $K'$. Then $X_1$ is homeomorphic to $X_1'$ by some homeomorphism $h_1$ \cite{Tsankov}
and we can define inductively a sequence of homeomorphisms $h_n:X_n\to X_n'$ for which $h_n\circ f_n=f_n' \circ h_{n+1}$. Consequently, the inverse limits $X_K$ and $X_{K'}$ are homeomorphic by the map arising as the inverse limits of homeomorphisms $h_n$.
On the other hand, if $X_K$ is homeomorphic to $X_{K'}$ then their nondegenerate components are mutually homeomomorhic which yields $K$ homeomorphic to $K'$.

Let us roughly argue that the assignment $K\mapsto X_K$ can be coded in a Borel way.
By the Arsenin-Kunugui Uniformization Theorem \cite[Theorem 35.46]{Kechris},
we can select a countable dense sequence of Borel maps $b_n: \mathcal K(Q)\to Q$ such that 
$\{b_n(K): n\geq k\}$ is a dense subset of $K$, for every $k\in\N$. 
Using this, the compactification $X_1$ can be realized as a subset 
\[\tilde X_1=\{(x, 0): x\in K\}\cup\{(b_n(K), 2^{-n}): n\in\N\}\]
of $Q\times [0,1]$. Then we replace each isolated point $(b_n(K), 2^{-n})$ of $\tilde X_1$ by a small copy of $K$ rescaled by the factor $2^{-n}$ and moved to a subset $\tilde K_n\subseteq Q\times \{2^{-n}\}$. Consequently, using the sequence $b_n(K)$, we approximate each $\tilde K_n\times\{0\}$ by a sequence in $Q\times [0,1]\times (0,1]$ to get $\tilde X_2\subseteq Q\times [0,1]^2$ as a relalization of $X_2$.
We can continue inductively by realizing $X_n$ as a subset $\tilde X_n$ of $Q\times [0,1]^n$ and finally the inverse limit $(X_n, f_n)$ corresponds to the closure of the set $\bigcup_{n\in\N} \tilde X_n \subseteq Q\times Q$, where we identify $[0,1]^n$ with $[0,1]^n\times\{0\}^\N\subseteq Q$.
\end{proof}

\begin{theorem}\label{conjugacycantoroidschaotic}
The conjugacy relation of chaotic systems on cantoroids is Borel bireducible with $E_{G_\infty}$.
\end{theorem}

\begin{proof}
We use the same trick as in Section 3, with some subtle differences (which includes the absence of Burgess theorem). 
Given a continuum $K$, consider the corresponding tame cantoroid $X_K$ given by the proof of Theorem \ref{tamecantoroids}. Consider the shift system $(X_K^{\Z}, \sigma_K)$. Note that the system is chaotic and the base space $X_K^{\Z}$ is a cantoroid since it is a countable product of cantoroids
(see \cite{BalibreaDownarowiczHricSnohaSpitalsky}). Moreover the map $K\mapsto (X_K^{\Z}, \sigma_K)$ is a natural Borel reduction. 
Since the homeomorphism relation of continua is Borel bireducible with $E_{G_{\infty}}$ it follows that $E_{G_{\infty}}$ is Borel reducible to the conjugacy relation of chaotic invertible systems on cantoroids.

On the other hand, by \cite{RosendalZielinski}, the conjugacy relation of (chaotic) systems on cantoroids is Borel reducible to $E_{G_\infty}$.
\end{proof}

\begin{definition}
For a compact space $X$ we define an equivalence relation $E_{tt}(X)$ on $X^{\N}$ for which a sequence $(a_n)_{n\in\N}$ is equivalent to $(a_n')_{n\in\N}$ if and only if a subsequence $(a_{n_k})_{k\in\N}$ is convergent iff so is $(a_{n_k}')_{k\in\N}$ for every $n_1<n_2<\dots$.
\end{definition}

\begin{definition}
For an equivalence relation $E$ on $X$ the equivalence relation $E^+$ is defined on $X^{\N}$ in such a way that $(a_n)_{n\in\N}$ and $(a_n')_{n\in\N}$ are $E^+$-equivalent iff 
\begin{align*}
\forall n\in\N\ \exists m\in\N&: (a_n,a_m')\in E \ \text{ and} \\
\forall m\in\N\ \exists n\in\N&: (a_n,a_m')\in E.    
\end{align*}
\end{definition}

\begin{theorem}\label{tamecantoroidsminimal}
The conjugacy relation of minimal invertible systems on tame cantoroids, which are not homeomorphic to the Cantor space, is Borel reducible to $E_{tt}(Q)^+$.
\end{theorem}

\begin{proof}
Let $\mathcal C$ be the collection of all tame cantoroids in $Q$ with infinitely many non-degenerate components (note that if a cantoroid, which is not homeomorphic to the Cantor space, admits a minimal homeomorphism, then it has infinitely many nondegenerate components).
First, we aim to find Borel maps which assigns to every $C\in\mathcal C$ a sequence $(K_n)_{n\in\N}$ of all components.
Consider the set 
\[\mathcal S:=\{(C, K):C\in\mathcal C, K \text{ is a nondegenerate component of } C\}.\]
The set $\mathcal S$ is Borel and it has countable infinite vertical sections. Hence by the Lusin-Novikov Uniformization Theorem \cite[Theorem 18.10]{Kechris} there is a sequence of Borel maps $b_n:\mathcal C\to\mathcal K(Q)$ such that $\{K: (C,K)\in \mathcal S\}=\{b_n(C): n\in\N\}$.
Let us fix a Borel selector $s:\mathcal K(Q)\to Q$
\cite[Theorem 12.13]{Kechris}.
Let 
\begin{align*}
f&: \{(C, h): C\in\mathcal C, h:C\to C\text{ is a minimal homeomorphism}\} \to (Q^\N)^\N, \\
f(C, h)&=\left(\bigl(h^k(s(b_n(C))))\bigr)_{k\in \N}: {n\in\N}\right).
\end{align*}
Clearly $f$ is Borel. Let us argue that $f$ is a reduction.

Suppose first that $(C, h)$ is conjugate to $(C', h')$ by $\varphi: C\to C'$. 
Then $\varphi(K)$ is a nondegenerate component of $C'$ and thus $\varphi(K)=b_m(C')$ for some $m\in\N$.
Since 
\begin{align*}
\varphi(h^k(s(K))) = h'^k(\varphi(s(K)))&\in h'^k(\varphi(K)), \\
h'^k(s(\varphi(K))))&\in h'^k(\varphi(K)),
\end{align*}
it follows that the distance of $\varphi(h^k(s(K)))$ and $h'^k(s(\varphi(K)))$ converges to 0 as $k\to\infty$. It follows that  both the sequences have convergent subsequences indexed by the same sequences of indices.
Since moreover $\varphi$ is a homeomorphism it follows that
 $h^k(s(K))_{k\in\N}$ and $h'^k(s(\varphi(K)))_{k\in\N}$ are $E_{tt}$-equivalent.
 Since we can go through the same argumentation starting with a component of $C'$,
 it follows that $f(C, h)$ and $f(C', h')$ are $E_{tt}(Q)^+$-equivalent.

For the opposite suppose that $f(C, h)$ and $f(C', h')$ are $E_{tt}(Q)^+$-equivalent. Then there is (at least one) nondegenerate component $K=b_n(C)$ of $C$.
By the definition of $E_{tt}(Q)^+$ it follows that there exists some $m\in\N$ for which the sequences $h'^k(s(b_m(C')))_{k\in\N}$ and $h^k(s(b_n(C)))_{k\in\N}$ are $E_{tt}$-equivalent.
Consequently, it is straightforward (and formally follows by \cite[Lemma 4.2]{LiPeng}) that $(C, h)$ and $(C', h')$ are conjugate.
\end{proof}

\section{Chaotic 1-dimensional systems}
Transitive invertible systems on the circle are classified by the rotation number and thus the conjugacy relation is smooth. On the interval, there are no transitive invertible systems available. Thus instead of invertible systems we consider noninvertible systems on the interval and on the circle in this section. It is well known that transitive interval systems are chaotic, i.e. have a dense set of periodic points (see e.g. \cite[Proposition 7.2]{Ruette}).

Let $I=[0,1]$. For any closed set $F\subseteq I$ we denote as $\Acc(F)$ the set of points $x\in F$ for which there is $\varepsilon>0$ such that $(x-\varepsilon,x)\cap F=\emptyset$ or $(x,x+\varepsilon)\cap F=\emptyset$. 
Let $\T$ be the collection of all transitive continuous self maps of $I$ and $E$ be the conjugacy relation on members of $\T$.
Assign to every $f\in \T$ the set
\[P_f=\bigcup_{n\in\N} \Acc(\Fix(f^n)).\]
The set $P_f$ does not contain all the periodic points necessarily, but it is dense (as well as the set of all periodic points). To see this it is enough to realize that 
$\Fix(f^n)$ has empty interior (if $\Fix(f^n)$ has nonempty interior, $f$ will be periodic on some interval which violates transitivity of $f$).
Consequently $\Acc(\Fix(f^n))$ is dense in $\Fix(f^n)$.

Let
\[F(f)=\{(f^n(x_i) ? f^m(x_j))_{i, j=1,\dots, k, m, n\in \N}: x_1,\dots,x_k\in P_f\},\]
where the symbol $?$ is a binary operation which attains naturally the values $\{<, =, >\}$.
Thus $F(f)$ is a countable subset of the Polish space 
$X=\bigcup_{k\in\N} \{<, =, >\}^{k\times k\times\N\times\N}$,
where $\{<, =, >\}$ is a discrete space with three points.

\begin{proposition}\label{conjugacyoftransintervalmaps}
Two maps $f, g\in \T$ are conjugate by a sense preserving homeomorphism if and only if $F(f)=F(g)$,
i.e. the map $F$ is a reduction of the conjugacy relation on $\T$ to the equality relation of countable subsets of $X$.
\end{proposition}

\begin{proof}
Clearly, any increasing homeomorphism $h$ of $[0,1]$, which conjugates $f, g\in\T$ has to map $\Fix f^n$ onto $\Fix g^n$ and thus also $\Acc(\Fix\ f^n)$ onto $\Acc(\Fix\ g^n)$. Thus $h$ maps $P_f$ onto $P_g$ and preserves the relations $<, =, >$. Thus $F(f)=F(g)$.

On the other hand, if $F(f)=F(g)$ for some $f, g\in\T$, then we may find a bijection $b: P_f\to P_g$ which preserves the order. Since both $P_f$ and $P_g$ are dense, there is an increasing homeomorphism $h:I\to I$ extending $b$. 

Since $b$ conjugates $f| P_f$ and $g| P_g$ it follows by continuity that $h$ conjugates $f$ and $g$.
\end{proof}


It remains to argue that there is a natural way to encode the range of $F$ so that we get a Borel reduction. This is done in the following proof. We use the notation $\N^{<\omega}$ for the set of all finite sequences of natural numbers.

\begin{theorem}\label{conjugacyinterval}
The sense-preserving conjugacy relation of transitive systems on the interval is Borel reducible to the universal $\Pi^0_3$-relation induced by $S_\infty$, known also as the equality of countable sets, i.e. the $E_=^+$ relation.
\end{theorem}

\begin{proof}
Consider the set 
\[M=\{(f, x)\in \T\times I: x\in P_f\}.\]

This is a Borel set with countable infinite vertical sections.
Hence, by Lusin-Novikov uniformization theorem \cite[Theorem 18.10]{Kechris},
there exist Borel maps $c_n: \T\to [0,1]$ such that 
the union of graphs of $c_n$ equals to the set $M$.
Define a map $G:\T\to X^{\N^{<\omega}}$ as follows:

\[G(f)= ((f^n(x_i) ? f^m(x_j))_{i, j=1,\dots, k, m, n\in \N}: x_1=c_{l_1}(f),\dots,x_k=c_{l_k}(f))_{(l_1,\dots, l_k)\in \N^{<\omega}}.\]
Clearly $G$ is a Borel map.
Note that the difference between $G$ and $F$ is that the range of $G$ is formed by countable indexed families whereas the range of $F$ is formed by countable sets. Also $F(f)=F(g)$ if and only if $G(f)=_+G(g)$.
Consequently, by Proposition \ref{conjugacyoftransintervalmaps}
we get that $G$ is a Borel reduction into $E_=^+$.


\end{proof}

\begin{theorem}\label{transitivecircle}
Conjugacy relation of transitive circle maps is Borel reducible to $E_{=^+}$.
\end{theorem}

\begin{proof}
We distinguish two kinds of transitive circle maps.
Those with periodic points and those without any periodic point.

Any transitive circle map without periodic points is a homeomorphism by \cite[Corollary 2]{AuslanderKatznelson}
and thus it is topologically conjugate to an irrational rotation. Here the rotation number (see e.g. \cite[Chapter 7.1]{BrinStuck}) is witnessing that this part is being smooth.

Thus it is enough to consider only transitive circle maps with at least one periodic point.
For such maps it is known that the set of all periodic points is dense \cite[Corollary 3.4]{CovenMulvey}.
Using a similar procedure as in the proof of the theorem above it can be proved that continuous circle maps are Borel reducible to $E_{=^+}$.
\end{proof}


By \cite{NagarSesha}, there are uncountably many pairwise nonconjugate transitive maps on $\R$. More precisely, the following is proved there.

\begin{theorem}\label{Nagar}
For every $a\in (\N\setminus\{1\})^{\Z}$ let $f_a:\R\to\R$ be a map for which 
\[f_a(x)=\begin{cases}
n-a_n, \text{ if } x=n\in\Z \text{ is odd}, \\
n+a_n, \text{ if } x=n\in\Z \text{ is even}
\end{cases} \]
and $f_a$ to be linear in every interval $[n,n+1]$, $n\in\Z$. 
Then each map $f_a$ is transitive and $f_a$ is conjugate to $f_b$ if and only if \begin{align}\label{eq:shift}
a=\sigma^k(b) \text{ for some } k\in\Z,
\end{align}

where $\sigma:\N^Z\to\N^\Z$ is the shift map.
\end{theorem}

\begin{corollary}\label{cor:chaoticinterval}
The relation $E_{0}$ is Borel reducible to the conjugacy relation of chaotic systems $[0,1]$.
\end{corollary}

\begin{proof}
Consider the maps $f_a$ given by Theorem \ref{Nagar}, but only for indices $a$ from the set $Z=(\{2,3\})^{\Z}$. Let $g_a:(0,1)\to (0,1)$ be a conjugate of $f_a$. Clearly, since $a$ is a bounded sequence, the map $f_a$ has limit $+\infty$ at $+\infty$ and $-\infty$ at $-\infty$. Consequently we may extend each $g_a$ to a continuous map $\bar g_a:[0,1]\to [0,1]$ so that $\bar g_a(0)=0$ and $\bar g_a(1)=1$. Note that each $\bar g_a$ remains transitive and that $\bar g_a$ is conjugate to $\bar g_b$ if and only if $f_a$ is conjugate to $f_b$.

On the set $Z$ we consider the equivalence relation given by $a\sim b$ if there exists $k\in \Z$ for which $a=\sigma^k(b)$ (see Equation \ref{eq:shift}).
The map $\varphi: a\mapsto \bar g_a$ is a reduction by Theorem \ref{Nagar}.
Since transitivity on the interval forces a dense set of periodic points, each of $\bar g_a$ is chaotic.
As a consequence of \cite[Exercise 6.1.7 and ]{Gao}, we get that that $E_0\leq_B E$ which completes the proof.
\end{proof}

\begin{remark}
We can obtain a similar result as in the previous Corollary also for circle maps. It would be enough to consider a one-point compactification of $(0,1)$ instead of the two-point compactification.
\end{remark}




\begin{remark}
There are no transitive homeomorphisms on non-degenerate dendrites. 
Suppose for contradiction that there is a dendrite $X$ with a transitive homeomorphisms $h$. Find a fixed point $x\in X$ of $h$. Clearly $X\setminus \{x\}$ might have only finitely many components, since otherwise transitivity of $h$ would be violated. So let $n\in\N$ be the number of components. Then $h$ permutes these components cyclically and thus $h^n$ maps each of this component into and also onto itself in a transitive way. Pick one of the components and let $C$ be its closure. Also let $f:=h^n|C$. Then $x$ is an end point of $C$. Let $A\subseteq C$ be any arc whose one end point is $x$. Then $B:=f(A)\cap A$ is again a nondegenerate arc with end point $x$ and either $f(B)\subseteq B$ or $f(B)\supseteq B$. We may suppose that $B$ is proper subset of $C$.
The set $S:=\{z\in C: [z,x]\cap (x,y)\neq\emptyset\}$ is a component of $C\setminus \{x,y\}$ and thus it is open. It follows that either $f(S)\subseteq S$ or $f(S)\supseteq S$. In any case we get a contradiction with transitivity of $f$.

On the other hand, every dendrite admits a chaotic map since even every locally connected continuum admits such a map by an unpublished result of Karasová and Vejnar (see \cite{AgronskyCeder} for a similar result on locally connected continua in $\R^n$.)
We propose that the conjugacy relation of chaotic systems on dendrites is Borel reducible to $E_{S_\infty}$.
\end{remark}

\section*{Acknowledgement}
The author is grateful to Dominik Kwietniak for some discussions on transitive maps on the interval.
\bibliographystyle{alpha}
\bibliography{citace.bbl}
\end{document}